\documentclass[11pt]{article}

\usepackage{amsmath}
\usepackage{amsthm}
\usepackage{amssymb}
\usepackage{latexsym}
\usepackage{mathabx}
\usepackage{pstricks}
\usepackage{graphicx, pst-plot, pst-node, pst-text, pst-tree}
\usepackage{titlefoot}
\usepackage{enumerate}
\usepackage{titlesec}
\usepackage{units}
\usepackage[small,it]{caption}

\usepackage{color}
\definecolor{lightgray}{rgb}{0.9, 0.9, 0.9}
\definecolor{darkgray}{rgb}{0.7, 0.7, 0.7}
\definecolor{darkblue}{rgb}{0, 0, .4}

\usepackage[bookmarks]{hyperref}
\hypersetup{
	colorlinks=true,
	linkcolor=darkblue,
	anchorcolor=darkblue,
	citecolor=darkblue,
	urlcolor=darkblue,
	pdfpagemode=UseThumbs,
	pdftitle={Permutation classes of every growth rate above 2.48188},
	pdfsubject={Combinatorics},
	pdfauthor={Vincent Vatter},
	pdfkeywords={growth rate, permutation class, Stanley-Wilf limit}
}

\newtheorem{theorem}{Theorem}[section]
\newtheorem{proposition}[theorem]{Proposition}

\newtheorem{conjecture}[theorem]{Conjecture}
\newtheorem{question}[theorem]{Question}
\newtheorem{problem}[theorem]{Problem}

\newcounter{todocounter}

\makeatletter
\newfont{\footsc}{cmcsc10 at 8truept}
\newfont{\footbf}{cmbx10 at 8truept}
\newfont{\footrm}{cmr10 at 10truept}
\renewenvironment{abstract}%
		{
		  \begin{list}{}%
		     {\setlength{\rightmargin}{1in}%
		      \setlength{\leftmargin}{1in}}%
		   \item[]\ignorespaces\begin{small}}%
		 {\end{small}\unskip\end{list}}

\datefoot{\today}
\amssubj{05A05, 05A15, 06A06}
\keywords{growth rate, permutation class, Stanley-Wilf limit}

\newpagestyle{main}[\small]{
	\headrule
	\sethead[\usepage][][]
	{\sc Permutation Classes of Every Growth Rate Above $2.48188$}{}{\usepage}}

\title{\sc{Permutation Classes of Every Growth Rate Above $2.48188$}}
\author{\sc{Vincent Vatter}\\
\small Department of Mathematics\\[-1pt]
\small Dartmouth College\\[-1pt]
\small Hanover, NH 03755\\[-10pt]}

\titleformat{\section}
	{\large\sc}
	{\thesection.}{1em}{}	

\titleformat{\subsection}
	{\normalsize\sc}
	{\thesubsection.}{1em}{}	

\date{}

\begin{document}
\maketitle

\pagestyle{main}

\newcommand{\Av}{\operatorname{Av}}
\newcommand{\Age}{\operatorname{Age}}
\newcommand{\A}{\mathcal{A}}
\newcommand{\C}{\mathcal{C}}
\newcommand{\D}{\mathcal{D}}
\newcommand{\E}{\mathcal{E}}
\newcommand{\HH}{\mathcal{H}}
\newcommand{\I}{\mathcal{I}}
\newcommand{\J}{\mathcal{J}}
\newcommand{\K}{\mathcal{K}}
\renewcommand{\L}{\mathcal{L}}
\newcommand{\M}{\mathcal{M}}
\newcommand{\N}{\mathcal{N}}
\renewcommand{\P}{\mathcal{P}}
\newcommand{\R}{\mathcal{R}}
\renewcommand{\S}{\mathcal{S}}
\renewcommand{\O}{\mathcal{O}}
\newcommand{\W}{\mathcal{W}}
\newcommand{\gr}{\mathrm{gr}}
\newcommand{\ip}{\operatorname{ip}}
\newcommand{\lgr}{\underline{\gr}}
\newcommand{\ugr}{\overline{\gr}}
\newcommand{\Grid}{\operatorname{Grid}}
\newcommand{\zpm}{0/\mathord{\pm} 1}
\newcommand{\proj}{\operatorname{proj}}
\newcommand{\height}{\operatorname{ht}}
\newcommand{\hjuxta}[2]{\left[\begin{array}{cc}#1&#2\end{array}\right]}

\begin{abstract}
We prove that there are permutation classes (hereditary properties of permutations) of every growth rate (Stanley-Wilf limit) at least $\lambda\approx 2.48187$, the unique real root of $x^5-2x^4-2x^2-2x-1$, thereby establishing a conjecture of Albert and Linton.
\end{abstract}

\section{Introduction}

The permutation $\pi$ of length $n$ {\it contains\/} the permutation $\sigma$ of length $k$, written $\sigma\le\pi$, if $\pi$ has a subsequence of length $k$ in the same relative order as $\sigma$.  For example, $\pi=391867452$ (written in list, or one-line notation) contains $\sigma=51342$, as can be seen by considering the subsequence $91672$ ($=\pi(2),\pi(3),\pi(5),\pi(6),\pi(9)$).  We further say that $\pi$ {\it properly contains\/} $\sigma$ if $\sigma\le\pi$ and $\sigma\neq\pi$.  A permutation class is a downset (or hereditary property) of permutations under this order; thus if $\C$ is a permutation class, $\pi\in\C$, and $\sigma\le\pi$ then $\sigma\in\C$.

We denote by $\C_n$ the set of permutations in $\C$ of length $n$.  The Marcus-Tardos Theorem~\cite{marcus:excluded-permut:} (formerly the Stanley-Wilf Conjecture) states that all permutation classes other than the class of all permutations grow at at most exponential speed, i.e., that these classes have finite {\it upper growth rates\/},
$$
\ugr(\C)=\limsup_{n\rightarrow\infty} \sqrt[n]{|\C_n|}.
$$
When $\lim_{n\rightarrow\infty}\sqrt[n]{|\C_n|}$ exists, which it does for all classes in this paper, we call it the {\it growth rate\/} of $\C$ and denote it by $\gr(\C)$.

Although vast amounts of research had considered the growth rates of {\it principally based\/} permutation classes, i.e., those defined by not containing a single permutation, Kaiser and Klazar~\cite{kaiser:on-growth-rates:} were the first to study growth rates of permutation classes in full generality.  They showed that permutation classes of growth rate less than $2$ satisfy one of the following:
\begin{itemize}
\item $|C_n|$ is polynomial for large $n$, or
\item $F_{n,k}\le |C_n|\le n^c F_{n,k}$ holds for integers $c\ge 0$ and $k\ge 2$, where $F_{n,k}$ denotes the $k$-generalized Fibonacci numbers defined by $F_{n,k}=F_{n-1,k}+F_{n-2,k}+\cdots+F_{n-k,k}$.
\end{itemize}
It follows that the growth rates of permutation classes less than $2$ are precisely the positive roots of $x^{k+1}-x^k-\cdots-x^2-1$ for some $k$.  

Vatter~\cite{vatter:small-permutati:} then showed that there are only countably many permutation classes of growth rate less than $\kappa$, the unique real root of $x^3-2x^2-1$, approximately $2.20557$, while there are uncountably many permutation classes of growth rate $\kappa$.  That work also characterizes the possible growth rates between $2$ and $\kappa$, which are roots of one of the following four families of polynomials
\begin{itemize}
\item $(x^3-2x^2-1)x^{k+1}-x+3$,
\item $(x^3-2x^2-1)x^{k+2}-x^2+2x+1$,
\item $(x^3-2x^2-1)x^{k+\ell}+x^\ell+1$, and
\item $(x^3-2x^2-1)x^k+1$,
\end{itemize}
for integers $k,\ell\ge 0$.  Note that the set of these growth rates contains no accumulation points from above, but does contain countably many accumulation points from below, and these accumulation points themselves accumulate at $\kappa$.

Results such as these have also been proved for a variety of combinatorial structures, for surveys of which we refer to Bollob\'as~\cite{bollobas:hereditary-and:} and Klazar~\cite{klazar:overview-of-som}.  In particular, Balogh, Bollob\'as, and Morris~\cite{balogh:hereditary-prop:ordgraphs} extended Kaiser and Klazar's work to the more general setting of ordered graphs\footnote{\label{fn-ordered-graphs}%
Let $G$ and $H$ be graphs on $\{1,\dots,n\}$ and $\{1,\dots,k\}$, respectively.  We say that $H$ is an {\it ordered subgraph\/} of $G$ if there is an increasing injection $f:\{1,\dots,k\}\rightarrow\{1,\dots,n\}$ such that $i\sim_H j$ if and only if $f(i)\sim_G f(j)$.  For any permutation $\pi$ of length $n$ we can construct the ordered graph $G_\pi$ on the vertex set $\{1,\dots,n\}$ by taking $i\sim_{G_\pi} j$ for $i<j$ if and only if $\pi(i)>\pi(j)$.  It follows that $G_\pi$ contains $G_\sigma$ as an ordered subgraph if and only if $\sigma\le\pi$, so every permutation class can be viewed as a hereditary property of ordered graphs.  Thus in particular, the set of growth rates of permutation classes is contained in the set of growth rates of hereditary properties of ordered graphs.  The converse, however, is not true.  For example, Balogh, Bollob\'as, and Morris~\cite{balogh:hereditary-prop:ordgraphs} show that the largest real root of $x^5-x^4-x^3-x^2-2x-1$, approximately $2.03166$, is the growth rate of a hereditary property of ordered graphs (they conjecture that it is the smallest such growth rate above $2$), but the results of Vatter~\cite{vatter:small-permutati:} show that this number is not the growth rate of a permutation class.}
and conjectured both that the set of growth rates of ordered graphs contains no accumulation points from above, and that all such growth rates are integers or algebraic irrationals.  These conjectures were disproved by Albert and Linton~\cite{albert:growing-at-a-pe:}, who showed that the set of growth rates of permutation classes contains a perfect set (a set equal to its accumulation points, e.g., the middle thirds Cantor set).  Albert and Linton also conjectured that there is some $\lambda$ such that every real number at least $\lambda$ is the growth rate of a permutation class.  Our main theorem, below, resolves this conjecture.

\begin{theorem}\label{thm-gr-all-large}
Let $\lambda$ denote the unique real root of $x^5-2x^4-2x^2-2x-1$, approximately $2.48187$.  Every real number greater than $\lambda$ is the growth rate of a permutation class.
\end{theorem}

In the next section we establish the terminology and basic (but technical) analytic facts needed for the proof of Theorem~\ref{thm-gr-all-large}, which is carried out in Section~\ref{sec-proof-main-thm}.  Before that, we conclude the introduction with a number of open problems.

Theorem~\ref{thm-gr-all-large} together with the results of Vatter~\cite{vatter:small-permutati:} leave us tantalizingly close to the ultimate aim of this line of research:

\begin{problem}
Characterize all growth rates of permutation classes.
\end{problem}

It would also be interesting to know, in a very rough sense, how many growth rates there are between $\kappa$ and $\lambda$:

\begin{question}
What is the Lebesgue measure of the set of growth rates of permutation classes in $[\kappa,\lambda]$?
\end{question}

Several remarkable phenomenon occur in the interval $[\kappa,\lambda]$: the first accumulation point from above, the appearance of the first perfect set, and the appearance of the first interval.  We have expended considerable effort optimizing $\lambda$, which we conjecture is best possible:

\begin{conjecture}
The set of growth rates of permutation classes below $\lambda$ is nowhere dense.
\end{conjecture}

\section{Sums of Permutations and Series of the Form $1/\left(1-\sum s_nx^n\right)$}

Given two permutations $\pi$ and $\sigma$ of respective lengths $n$ and $k$, their {\it direct sum\/} (or simply {\it sum\/}), $\pi\oplus\sigma$, is the permutation of length $n+k$ which consists of $\pi$ followed by a shifted copy of $\sigma$:
$$
(\pi\oplus\sigma)(i)=\left\{
\begin{array}{ll}
\pi(i)&\mbox{for $i\le n$,}\\
\sigma(i-n)+n&\mbox{for $n+1\le i\le n+k$.}
\end{array}
\right.
$$
A class $\C$ is said to be {\it sum closed\/} if $\pi\oplus\sigma\in\C$ whenever $\pi,\sigma\in\C$.  The classes we use to prove Theorem~\ref{thm-gr-all-large} are all sum closed and so we begin by recalling a few simple facts about these classes.

A permutation is said to be {\it sum indecomposable\/} (or {\it connected\/} or simply {\it indecomposable\/}) if it cannot be written as the direct sum of two shorter permutations.  Note that every permutation has a unique representation as a sum of sum indecomposable permutations.

\begin{proposition}\label{prop-gr-sums}
For all $n\ge 1$, let $s_n$ denote the number of sum indecomposable permutations in the sum closed class $\C$.  Then the generating function for $\C$ is $1/\left(1-\sum s_nx^n\right)$.
\end{proposition}
\begin{proof}
The permutations in $\C$ are in bijection with sequences of nonempty sum indecomposable permutations in $\C$.  Thus the generating function for $\C$ is $1+\left(s_1x+s_2x^2+\cdots\right)+\left(s_1x+s_2x^2+\cdots\right)^2+\cdots=1/\left(1-\sum s_nx^n\right)$.
\end{proof}


Since the classes we use to prove the main result are sum closed, Proposition~\ref{prop-gr-sums} shows that we may limit our attention to series of the form $1/\left(1-\sum s_nx^n\right)$, thinking of $(s_n)$ as the sequence enumerating sum indecomposables in a sum closed permutation class.  The remainder of this section is devoted to studying generating functions of this form, and our main result, Proposition~\ref{prop-gr-intervals}, provides conditions on when the growth rates of a collection of sum closed permutation classes form an interval.  In the following section, we construct sum closed classes whose sequences of sum indecomposable elements satisfy the hypotheses of this proposition, thereby proving Theorem~\ref{thm-gr-all-large}.

First we extend the notion of growth rates to generating functions, by defining
\begin{eqnarray*}
\ugr\left(\sum a_nx^n\right)&=&\limsup_{n\rightarrow\infty}\sqrt[n]{a_n}\mbox{ and}\\
\gr\left(\sum a_nx^n\right)&=&\lim_{n\rightarrow\infty}\sqrt[n]{a_n},
\end{eqnarray*}
when this limit exists (which we will always be the case for the series we consider).  First recall the following fact.

\newtheorem*{pringsheimsthm}{Pringsheim's Theorem}
\begin{pringsheimsthm}[see Flajolet and Sedgewick~{\cite[Section IV.3]{flajolet:analytic-combin:}}]
For any sequence $(a_n)$ of nonnegative numbers, the upper growth rate of $\sum a_nx^n$ is equal to the reciprocal of its smallest positive pole.
\end{pringsheimsthm}

Our next two propositions can be found (in slightly different forms) in Albert and Linton's work~\cite{albert:growing-at-a-pe:}.

\begin{proposition}\label{prop-gr-seq}
For any sequence $(s_n)_{n\ge 1}$ of bounded positive integers, the growth rate of $1/(1-\sum s_nx^n)$ exists and is equal to the unique positive solution of $\sum s_nx^{-n}=1$.
\end{proposition}
\begin{proof}
First we prove that the limit exists.  Consider an alphabet $A$ containing, for each $n$, $s_n$ letters of weight $x^n$.  Then the generating function counting words over $A$ by weight is $1/(1-\sum s_nx^n)$; suppose that this is equal to $\sum a_nx^n$.  The concatenation operation gives an injection from pairs of words of weight $m$ and $n$ to words of weight $m+n$, proving that $a_{m+n}\ge a_ma_n$.  Because $s_1\ge 1$, we have that $a_n\ge 1$ for all $n$, so $\log a_n$ exists for all $n$ and is superadditive.  Fekete's lemma then implies that
$$
\lim_{n\rightarrow\infty} \sqrt[n]{a_n}=e^{\lim\log{a_n}/n}=e^{\sup \log{a_n}/n}.
$$
(This part of the proof is essentially due to Arratia~\cite{arratia:on-the-stanley-:}.)

Now let $\psi$ denote the growth rate of $1/\left(1-\sum s_nx^n\right)$.  By Pringsheim's Theorem $\psi$ is the reciprocal of the smallest positive pole of this series, and because $(s_n)$ is bounded, the radius of convergence of $\sum s_nx^n$ is $1$.  Thus the smallest pole of $1/(1-\sum s_nx^n)$ is the least positive solution of $\sum s_nx^n=1$, making $\psi$ the greatest positive solution of $\sum s_nx^{-n}=1$.  Finally, note that $\sum s_nx^{-n}$ is decreasing for positive $x$, so $\psi$ is actually the unique positive solution of $\sum s_nx^{-n}=1$.
\end{proof}

In particular, since $\displaystyle\sum_{n\ge 1}\frac{c}{(c+1)^n}=1$, the growth rate of $\displaystyle\frac{1}{1-\sum_{n\ge 1} cx^n}$ is $c+1$.

\begin{proposition}\label{prop-gr-difference}
Fix $\epsilon>0$ and $c$ a positive integer.  There is a constant $m=m(\epsilon,c)$ such that if two sequences $(r_n)_{n\ge 1}$ and $(s_n)_{n\ge 1}$ of positive integers each at most $c$ agree up to their $m$th terms then the growth rates of $1/\left(1-\sum r_nx^n\right)$ and $1/\left(1-\sum s_nx^n\right)$ are within $\epsilon$ of each other.
\end{proposition}
\begin{proof}
Let $\rho=\gr\left(1/\left(1-\sum r_nx^n\right)\right)$, $\psi=\gr\left(1/\left(1-\sum s_nx^n\right)\right)$, and suppose that $\rho<\psi$.  As $1\le r_n,s_n\le c$, we have that
$$
2\le \rho<\psi\le\gr\left(\frac{1}{1-\sum cx^n}\right)=c+1,
$$
and so we may choose $m$ such that
$$
\left|\sum_{n\ge m+1} s_n\psi^{-n}-\sum_{n\ge m+1}r_n\rho^{-n}\right|<\epsilon/(c+1)^2.
$$
By Proposition~\ref{prop-gr-seq}, we see that
$$
\sum_{n\ge 1}r_n\rho^{-n}-\sum_{n\ge 1}s_n\psi^{-n}=1-1=0,
$$
so if $(r_n)$ and $(s_n)$ agree up to their $m$th terms then
$$
\sum_{n=1}^m r_n(\rho^{-n}-\psi^{-n})
=
\sum_{n\ge m+1} s_n\psi^{-n}-\sum_{n\ge m+1}r_n\rho^{-n}.
$$
Because $\rho<\psi$, the sum in the left-hand side of the above equation contains only positive terms.  Thus we have in particular that $\rho^{-1}-\psi^{-1}<\epsilon/(c+1)^2$, implying that $\psi-\rho<\epsilon\rho\psi/(c+1)^2<\epsilon$, as desired.
\end{proof}

Our next result shows how to construct intervals of growth rates given sufficient flexibility in sum indecomposable sequences.  In it we write $(r_n)\preceq (t_n)$ if $(r_n)$ {\it is dominated by\/} $(t_n)$, i.e., if $r_n\le t_n$ for all $n\ge 1$.

\begin{proposition}\label{prop-gr-intervals}
Let $(r_n)\preceq (t_n)$ be bounded sequences of positive integers such that $t_n=r_n$ for all $n<k$ and $t_n-r_n\ge b-1$ for all $n\ge k$ (for some $k$).  Then every real number between
\begin{itemize}
\item $\gr\left(1/\left(1-\sum r_nx^n\right)\right)$ and
\item $\min\{b, \gr\left(1/\left(1-\sum t_nx^n\right)\right)\}$
\end{itemize}
is equal to $\gr\left(1/\left(1-\sum s_nx^n\right)\right)$ for some sequence $(s_n)$ of positive integers satisfying $(r_n)\preceq (s_n)\preceq (t_n)$.
\end{proposition}
\begin{proof}
Let $\rho=\gr\left(1/\left(1-\sum r_nx^n\right)\right)$, $\tau=\gr\left(1/\left(1-\sum t_nx^n\right)\right)$, and choose $\gamma$ between $\rho$ and $\min\{b,\tau\}$.

For each $n$ choose $s_n$ to be the greatest integer between $r_n$ and $t_n$ such that the growth rate of
$$
\frac{1}{1-s_1x-\cdots-s_nx^n-r_{n+1}x^{n+1}-r_{n+2}x^{n+2}-\cdots}
$$
is at most $\gamma$.  If follows that $\psi=\gr\left(1/\left(1-\sum s_nx^n\right)\right)\le\gamma$; we seek to prove that it is precisely $\gamma$.

Suppose to the contrary that this does not hold, so $\psi=\gamma-\epsilon$ for some $\epsilon>0$.  By Proposition~\ref{prop-gr-difference}, there is an $m$ such that
$$
\left|\gr\left(\frac{1}{1-s_1x-\cdots-s_mx^m-t_{m+1}x^{m+1}-t_{m+2}x^{m+2}-\cdots}\right)-\psi\right|<\epsilon,
$$
so by our choice of $(s_n)$, $s_n=t_n$ for all sufficiently large $n$.  If $s_n=t_n$ for all $n$ then we have a contradiction, as $\gamma=\psi=\tau$.  Thus we may choose $m>k$ minimal so that $s_n=t_n$ for all $n>m$ (and $s_m<t_m$).  Thus we have, from the definition of $(s_n)$, that
$$
\sum_{n\ge 1} s_n\gamma^{-n}
=
\sum_{n=1}^m s_n\gamma^{-n}
+
\sum_{n\ge m+1} t_n\gamma^{-n}
< 1.
$$
However, because $\gamma\le b$,
$$
\gamma^{-m}
\le
\frac{(b-1)\gamma^{-m-1}}{1-\gamma^{-1}}
=
\sum_{i\ge m+1} (b-1)\gamma^{-m}
\le
\sum_{i\ge m+1} (t_m-r_m)\gamma^{-m},
$$
and thus
$$
\sum_{n=1}^m s_n\gamma^{-n}
+
\gamma^{-m}
+
\sum_{n\ge m+1} r_i\gamma^{-n}
<1,
$$
so the growth rate of
$$
\frac{1}{1-s_1x-\cdots-s_{k-1}x^k-(s_k+1)x^k-r_{k+1}x^{k+1}-r_{k+2}t^{k+2}-\cdots}
$$
is less than $\gamma$, contradicting our choice of $s_k$ and completing the proof.
\end{proof}

We conclude this section by noting that perfect sets of growth rates can be constructed with much weaker hypotheses:

\begin{proposition}\label{prop-gr-perfect}
Let $(r_n)\preceq (t_n)$ be bounded sequences of positive integers such that $r_n\neq t_n$ for infinitely many values of $n$.  Then the set $S$ of growth rates of series of the form $1/\left(1-\sum s_nx^n\right)$ for sequences $(s_n)$ of positive integers satisfying $(r_n)\preceq(s_n)\preceq(t_n)$ is perfect.
\end{proposition}
\begin{proof}
We need to prove that the set of accumulation points of $S$ is precisely $S$.  That every point of $S$ is an accumulation point follows from Proposition~\ref{prop-gr-difference} and our hypothesis that $r_n\neq t_n$ infinitely often.

For the other direction, suppose that $\gamma$ is an accumulation point of $S$.  Thus there are sequences $\left(s_n^{(1)}\right)$, $\left(s_n^{(2)}\right)$, and so on such that the growth rates of the corresponding series converge to $\gamma$.  Since $(t_n)$ is bounded, $\left(s_n^{(1)}\right),\left(s_n^{(2)}\right),\dots$ contains an infinite subsequence with the same first $m$ entries for every $m$.  This, together with Proposition~\ref{prop-gr-difference}, implies that $\gamma$ lies in $S$.
\end{proof}

\section{Proof of the Main Theorem}\label{sec-proof-main-thm}

In order to prove that the set of growth rates of permutation classes contains an interval, and to eventually prove Theorem~\ref{thm-gr-all-large}, Proposition~\ref{prop-gr-intervals} shows that we need to construct classes with a wide range of sequences of sum indecomposable permutations.  To do so, we exhibit infinite antichains of sum indecomposable permutations (in other words, infinite sets of sum indecomposable permutations such that none is contained in another).  It follows that a class can contain any subset of these antichains, giving the flexibility we desire.  In order to introduce these antichains we need more terminology.

An {\it interval\/} of the permutation $\pi$ is a set of contiguous indices $I=[a,b]$ such that the set of values $\pi(I)=\{\pi(i):i\in I\}$ is also contiguous.  Every permutation of length $n$ has intervals of length $0$, $1$, and $n$, and permutations with no other intervals are called {\it simple\/}.  Given a permutation $\sigma\in S_m$ and nonempty permutations $\alpha_1,\dots,\alpha_m$, the {\it inflation\/} of $\sigma$ by $\alpha_1,\dots,\alpha_m$, denoted $\sigma[\alpha_1,\dots,\alpha_m]$, is the permutation obtained by replacing each entry $\sigma(i)$ of $\sigma$ with an interval that is in the same relative order as $\alpha_i$.  For example, $3142[132,21,1,123]=687\ 21\ 9\ 345$ (see Figure~\ref{fig-int-and-osc}), while $\pi\oplus\sigma=12[\pi,\sigma]$ for all $\pi$ and $\sigma$.

A simple permutation is therefore a permutation that cannot be expressed as the inflation of a shorter permutation of length greater than $1$, and conversely,

\begin{proposition}[Albert and Atkinson~\cite{albert:simple-permutat:}]\label{simple-decomp-1}
Every permutation $\pi\neq 1$ is the inflation of a unique simple permutation of length at least $2$, known as the {\it simple quotient\/} of $\pi$.
\end{proposition}

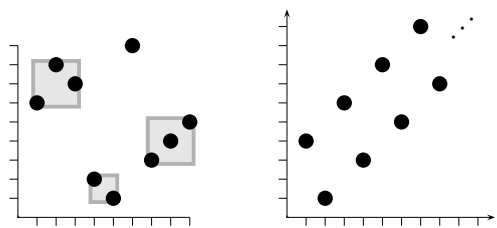
\begin{figure}
\begin{center}
\begin{tabular}{ccc}
\psset{xunit=0.01in, yunit=0.01in}
\psset{linewidth=0.005in}
\begin{pspicture}(0,0)(90,90)
\psaxes[dy=10,Dy=1,dx=10,Dx=1,tickstyle=bottom,showorigin=false,labels=none](0,0)(90,90)
\psframe[linecolor=darkgray,fillstyle=solid,fillcolor=lightgray,linewidth=0.02in](7,57)(33,83)
\psframe[linecolor=darkgray,fillstyle=solid,fillcolor=lightgray,linewidth=0.02in](37,7)(53,23)
\psframe[linecolor=darkgray,fillstyle=solid,fillcolor=lightgray,linewidth=0.02in](67,27)(93,53)
\pscircle*(10,60){0.04in}
\pscircle*(20,80){0.04in}
\pscircle*(30,70){0.04in}
\pscircle*(40,20){0.04in}
\pscircle*(50,10){0.04in}
\pscircle*(60,90){0.04in}
\pscircle*(70,30){0.04in}
\pscircle*(80,40){0.04in}
\pscircle*(90,50){0.04in}
\end{pspicture}
&\rule{10pt}{0pt}&
\psset{xunit=0.01in, yunit=0.01in}
\psset{linewidth=0.005in}
\begin{pspicture}(0,0)(110,110)
\psaxes[dy=10, Dy=1, dx=10, Dx=1,tickstyle=bottom,showorigin=false,labels=none]{->}(0,0)(109,109)
\pscircle*(10,40){4.0\psxunit}
\pscircle*(20,10){4.0\psxunit}
\pscircle*(30,60){4.0\psxunit}
\pscircle*(40,30){4.0\psxunit}
\pscircle*(50,80){4.0\psxunit}
\pscircle*(60,50){4.0\psxunit}
\pscircle*(70,100){4.0\psxunit}
\pscircle*(80,70){4.0\psxunit}
\pstextpath[c]{\psline[linecolor=white](85,98)(95,108)}{$\dots$}
\end{pspicture}
\end{tabular}
\end{center}
\caption{On the left, the plot of the permutation $687219345=3142[132,21,1,123]$.  On the right, the plot of the increasing oscillating sequence.}\label{fig-int-and-osc}
\end{figure}

We define the {\it increasing oscillating sequence\/} as the sequence
$$
4,1,6,3,8,5,\dots,2k+2,2k-1,\dots.
$$
A plot is shown in Figure~\ref{fig-int-and-osc}.  We further define an {\it increasing oscillation\/} to be any simple permutation that is contained in the increasing oscillating sequence.

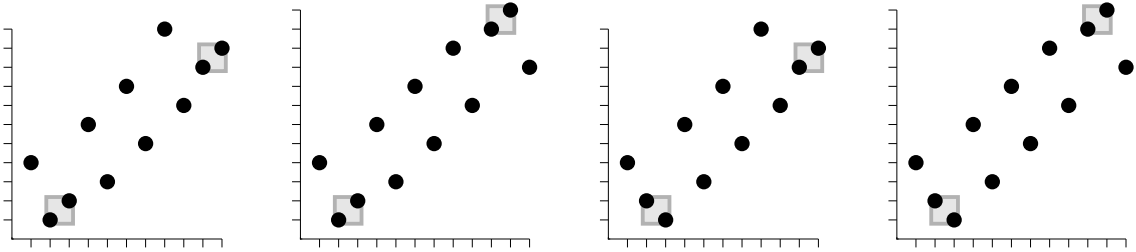
\begin{figure}
\begin{center}
\begin{tabular}{ccccccc}
\psset{xunit=0.01in, yunit=0.01in}
\psset{linewidth=0.005in}
\begin{pspicture}(0,0)(110,110)
\psaxes[dy=10,Dy=1,dx=10,Dx=1,tickstyle=bottom,showorigin=false,labels=none](0,0)(110,110)
\psframe[linecolor=darkgray,fillstyle=solid,fillcolor=lightgray,linewidth=0.02in](17,7)(33,23)
\psframe[linecolor=darkgray,fillstyle=solid,fillcolor=lightgray,linewidth=0.02in](97,87)(113,103)
\pscircle*(10,40){0.04in}
\pscircle*(20,10){0.04in}
\pscircle*(30,20){0.04in}
\pscircle*(40,60){0.04in}
\pscircle*(50,30){0.04in}
\pscircle*(60,80){0.04in}
\pscircle*(70,50){0.04in}
\pscircle*(80,110){0.04in}
\pscircle*(90,70){0.04in}
\pscircle*(100,90){0.04in}
\pscircle*(110,100){0.04in}
\end{pspicture}
&
\rule{3pt}{0pt}
&
\psset{xunit=0.01in, yunit=0.01in}
\psset{linewidth=0.005in}
\begin{pspicture}(0,0)(120,120)
\psaxes[dy=10,Dy=1,dx=10,Dx=1,tickstyle=bottom,showorigin=false,labels=none](0,0)(120,120)
\psframe[linecolor=darkgray,fillstyle=solid,fillcolor=lightgray,linewidth=0.02in](17,7)(33,23)
\psframe[linecolor=darkgray,fillstyle=solid,fillcolor=lightgray,linewidth=0.02in](97,107)(113,123)
\pscircle*(10,40){0.04in}
\pscircle*(20,10){0.04in}
\pscircle*(30,20){0.04in}
\pscircle*(40,60){0.04in}
\pscircle*(50,30){0.04in}
\pscircle*(60,80){0.04in}
\pscircle*(70,50){0.04in}
\pscircle*(80,100){0.04in}
\pscircle*(90,70){0.04in}
\pscircle*(100,110){0.04in}
\pscircle*(110,120){0.04in}
\pscircle*(120,90){0.04in}
\end{pspicture}
&
\rule{3pt}{0pt}
&
\psset{xunit=0.01in, yunit=0.01in}
\psset{linewidth=0.005in}
\begin{pspicture}(0,0)(110,110)
\psaxes[dy=10,Dy=1,dx=10,Dx=1,tickstyle=bottom,showorigin=false,labels=none](0,0)(110,110)
\psframe[linecolor=darkgray,fillstyle=solid,fillcolor=lightgray,linewidth=0.02in](17,7)(33,23)
\psframe[linecolor=darkgray,fillstyle=solid,fillcolor=lightgray,linewidth=0.02in](97,87)(113,103)
\pscircle*(10,40){0.04in}
\pscircle*(20,20){0.04in}
\pscircle*(30,10){0.04in}
\pscircle*(40,60){0.04in}
\pscircle*(50,30){0.04in}
\pscircle*(60,80){0.04in}
\pscircle*(70,50){0.04in}
\pscircle*(80,110){0.04in}
\pscircle*(90,70){0.04in}
\pscircle*(100,90){0.04in}
\pscircle*(110,100){0.04in}
\end{pspicture}
&
\rule{3pt}{0pt}
&
\psset{xunit=0.01in, yunit=0.01in}
\psset{linewidth=0.005in}
\begin{pspicture}(0,0)(120,120)
\psaxes[dy=10,Dy=1,dx=10,Dx=1,tickstyle=bottom,showorigin=false,labels=none](0,0)(120,120)
\psframe[linecolor=darkgray,fillstyle=solid,fillcolor=lightgray,linewidth=0.02in](17,7)(33,23)
\psframe[linecolor=darkgray,fillstyle=solid,fillcolor=lightgray,linewidth=0.02in](97,107)(113,123)
\pscircle*(10,40){0.04in}
\pscircle*(20,20){0.04in}
\pscircle*(30,10){0.04in}
\pscircle*(40,60){0.04in}
\pscircle*(50,30){0.04in}
\pscircle*(60,80){0.04in}
\pscircle*(70,50){0.04in}
\pscircle*(80,100){0.04in}
\pscircle*(90,70){0.04in}
\pscircle*(100,110){0.04in}
\pscircle*(110,120){0.04in}
\pscircle*(120,90){0.04in}
\end{pspicture}
\end{tabular}
\end{center}
\caption{Members of the infinite antichains $U^{12,12}$ (left) and $U^{21,12}$ (right).}\label{fig-first-antichains}
\end{figure}

We are now ready to describe the infinite antichains of sum indecomposable permutations that we use.  For each even $k\ge 4$, let $\sigma_k$ denote the increasing oscillation in the same relative order as the first $k$ entries of the increasing oscillating sequence, while for each odd $k\ge 5$ let $\sigma_k$ denote the increasing oscillation in the same relative order as the least $k$ entries of this sequence.  The elements of the antichain $U^{\alpha,\beta}$ are formed by inflating the least entry of $\sigma_k$ by $\alpha$ and the greatest entry by $\beta$ when $k$ is even, and when $k$ is odd, inflating the least entry of $\sigma_k$ by $\alpha$ and the rightmost entry by $\beta$.  Figure~\ref{fig-first-antichains} shows two examples.  We begin by observing that these are indeed antichains:

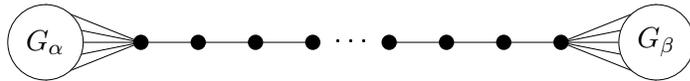
\begin{figure}
\begin{center}
\psset{xunit=0.01in, yunit=0.01in}
\psset{linewidth=0.005in}
\begin{pspicture}(0,0)(360,40)
\pscircle*(70,20){0.04in}
\pscircle*(100,20){0.04in}
\pscircle*(130,20){0.04in}
\pscircle*(160,20){0.04in}
\rput[c](181,20){$\cdots$}
\pscircle*(200,20){0.04in}
\pscircle*(230,20){0.04in}
\pscircle*(260,20){0.04in}
\pscircle*(290,20){0.04in}
\psline(290,20)(340,40)
\psline(290,20)(340,30)
\psline(290,20)(340,10)
\psline(290,20)(340,0)
\psline(200,20)(340,20)
\pscircle*[fillcolor=white,linecolor=white](340,20){20\psxunit}
\pscircle(340,20){20\psxunit}
\rput[c](340,20){$G_\beta$}
\psline(20,40)(70,20)
\psline(20,30)(70,20)
\psline(20,10)(70,20)
\psline(20,0)(70,20)
\psline(20,20)(160,20)
\pscircle*[fillcolor=white,linecolor=white](20,20){20\psxunit}
\pscircle(20,20){20\psxunit}
\rput[c](20,20){$G_\alpha$}
\end{pspicture}
\end{center}
\caption{The (unordered) graph corresponding to an element of the $U^{\alpha,\beta}$ antichain.}\label{fig-U-graphs}
\end{figure}

\begin{proposition}\label{prop-U-antichain}
For all $\alpha,\beta\neq 1$, $U^{\alpha,\beta}$ is an antichain.  Moreover, if $\alpha,\alpha',\beta,\beta'\neq 1$ and either $\alpha$ and $\alpha'$ are incomparable or $\beta$ and $\beta'$ are incomparable then $U^{\alpha,\beta}\cup U^{\alpha',\beta'}$ forms an antichain.
\end{proposition}
\begin{proof}
The proposition follows easily when translated to the equivalent fact about ordered graphs, as defined in Footnote~\ref{fn-ordered-graphs}.  Suppose that $U^{\alpha,\beta}$ failed to be an antichain.  Then there is a comparison between two of the ordered graphs corresponding to the members of $U^{\alpha,\beta}$, and thus a comparison between the unordered versions of these graphs.  These unordered graphs, depicted in Figure~\ref{fig-U-graphs}, clearly form an antichain.  Similar reasoning, but using the ordered versions of these graphs, shows that $U^{\alpha,\beta}\cup U^{\alpha',\beta'}$ is an antichain whenever $\alpha$ and $\alpha'$ are incomparable or $\beta$ and $\beta'$ are incomparable.
\end{proof}

We are now ready to establish our first interval of growth rates.

\begin{proposition}\label{prop-gr-interval1}
Every real number between $\lambda$ and the greatest positive root of $x^5-2x^4-2x^2-10x+5$ ($\approx 2.69284$) is the growth rate of a permutation class.
\end{proposition}
\begin{proof}
Let $A=U^{12,12}\cup U^{21,12}$.  It follows from Proposition~\ref{prop-U-antichain} that $A$ is an antichain, and it is clear from inspection that $A$ contains precisely two sum indecomposable permutations of each length $n\ge 5$.  We now need to enumerate the sum indecomposable permutations which are properly contained in $A$.  For this observe that removing any point from the ``middle'' of a $U$ antichain member results in a sum decomposable permutation.  Thus the sum indecomposable permutations properly contained in $A$ are all obtained by removing points at the ends.  For $n\ge 6$ these split into six chains, each containing one element of each length, see Figure~\ref{fig-gr1-prop-contain} for the $n=8$ case.  The six chains can roughly be described as: the ``beginnings'' of $U^{12,12}$ and $U^{21,12}$ (there are two such chains), the ``endings'' of $U^{12,12}$ and $U^{21,12}$ (there are two such chains), and two increasing oscillations.

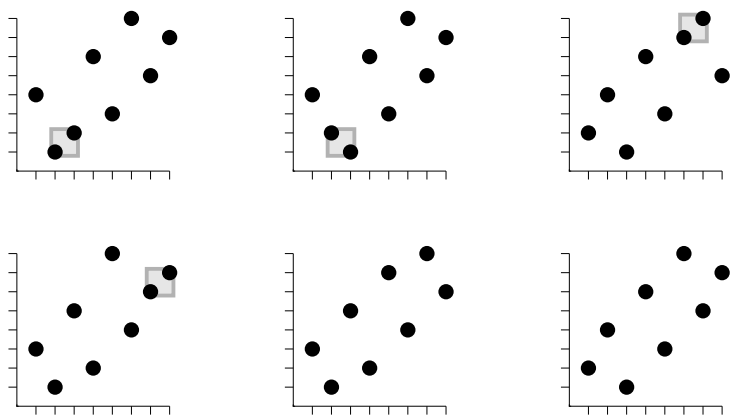
\begin{figure}
\begin{center}
\begin{tabular}{ccccc}

\psset{xunit=0.01in, yunit=0.01in}
\psset{linewidth=0.005in}
\begin{pspicture}(0,0)(80,80)
\psaxes[dy=10,Dy=1,dx=10,Dx=1,tickstyle=bottom,showorigin=false,labels=none](0,0)(80,80)
\psframe[linecolor=darkgray,fillstyle=solid,fillcolor=lightgray,linewidth=0.02in](17,7)(33,23)
\pscircle*(10,40){0.04in}
\pscircle*(20,10){0.04in}
\pscircle*(30,20){0.04in}
\pscircle*(40,60){0.04in}
\pscircle*(50,30){0.04in}
\pscircle*(60,80){0.04in}
\pscircle*(70,50){0.04in}
\pscircle*(80,70){0.04in}
\end{pspicture}

&\rule{20pt}{0pt}&

\psset{xunit=0.01in, yunit=0.01in}
\psset{linewidth=0.005in}
\begin{pspicture}(0,0)(80,80)
\psaxes[dy=10,Dy=1,dx=10,Dx=1,tickstyle=bottom,showorigin=false,labels=none](0,0)(80,80)
\psframe[linecolor=darkgray,fillstyle=solid,fillcolor=lightgray,linewidth=0.02in](17,7)(33,23)
\pscircle*(10,40){0.04in}
\pscircle*(20,20){0.04in}
\pscircle*(30,10){0.04in}
\pscircle*(40,60){0.04in}
\pscircle*(50,30){0.04in}
\pscircle*(60,80){0.04in}
\pscircle*(70,50){0.04in}
\pscircle*(80,70){0.04in}
\end{pspicture}

&\rule{20pt}{0pt}&

\psset{xunit=0.01in, yunit=0.01in}
\psset{linewidth=0.005in}
\begin{pspicture}(0,0)(80,80)
\psaxes[dy=10,Dy=1,dx=10,Dx=1,tickstyle=bottom,showorigin=false,labels=none](0,0)(80,80)
\psframe[linecolor=darkgray,fillstyle=solid,fillcolor=lightgray,linewidth=0.02in](57,67)(73,83)
\pscircle*(10,20){0.04in}
\pscircle*(20,40){0.04in}
\pscircle*(30,10){0.04in}
\pscircle*(40,60){0.04in}
\pscircle*(50,30){0.04in}
\pscircle*(60,70){0.04in}
\pscircle*(70,80){0.04in}
\pscircle*(80,50){0.04in}
\end{pspicture}

\\\\\\

\psset{xunit=0.01in, yunit=0.01in}
\psset{linewidth=0.005in}
\begin{pspicture}(0,0)(80,80)
\psaxes[dy=10,Dy=1,dx=10,Dx=1,tickstyle=bottom,showorigin=false,labels=none](0,0)(80,80)
\psframe[linecolor=darkgray,fillstyle=solid,fillcolor=lightgray,linewidth=0.02in](67,57)(83,73)
\pscircle*(10,30){0.04in}
\pscircle*(20,10){0.04in}
\pscircle*(30,50){0.04in}
\pscircle*(40,20){0.04in}
\pscircle*(50,80){0.04in}
\pscircle*(60,40){0.04in}
\pscircle*(70,60){0.04in}
\pscircle*(80,70){0.04in}
\end{pspicture}

&\rule{20pt}{0pt}&

\psset{xunit=0.01in, yunit=0.01in}
\psset{linewidth=0.005in}
\begin{pspicture}(0,0)(80,80)
\psaxes[dy=10,Dy=1,dx=10,Dx=1,tickstyle=bottom,showorigin=false,labels=none](0,0)(80,80)
\pscircle*(10,30){0.04in}
\pscircle*(20,10){0.04in}
\pscircle*(30,50){0.04in}
\pscircle*(40,20){0.04in}
\pscircle*(50,70){0.04in}
\pscircle*(60,40){0.04in}
\pscircle*(70,80){0.04in}
\pscircle*(80,60){0.04in}
\end{pspicture}

&\rule{20pt}{0pt}&

\psset{xunit=0.01in, yunit=0.01in}
\psset{linewidth=0.005in}
\begin{pspicture}(0,0)(80,80)
\psaxes[dy=10,Dy=1,dx=10,Dx=1,tickstyle=bottom,showorigin=false,labels=none](0,0)(80,80)
\pscircle*(10,20){0.04in}
\pscircle*(20,40){0.04in}
\pscircle*(30,10){0.04in}
\pscircle*(40,60){0.04in}
\pscircle*(50,30){0.04in}
\pscircle*(60,80){0.04in}
\pscircle*(70,50){0.04in}
\pscircle*(80,70){0.04in}
\end{pspicture}

\end{tabular}

\end{center}
\caption{The six sum indecomposable permutations of length $8$ properly contained in the antichain $A$ of Proposition~\ref{prop-gr-interval1}.}\label{fig-gr1-prop-contain}
\end{figure}

The smaller cases can be checked by hand, and thus the set of sum indecomposable permutations properly contained in $A$ is counted by the sequence $(r_n)=1,1,3,5,6,6,6,\dots$, while the set of sum indecomposable permutations contained (either properly or otherwise) in any member of $A$ is counted by $1,1,3,5,8,8,8,\dots$.  We now add to $A$ the $8$ sum indecomposable permutations of length $4$ that are not contained in it (there are $13$ sum indecomposable permutations of length $4$ in total), giving the infinite antichain $A'$ whose sum indecomposable permutations are enumerated by $(t_n)=1,1,3,13,8,8,8,\dots$.

Therefore, all sequences $(s_n)$ satisfying $(r_n)\preceq (s_n)\preceq (t_n)$ for
$$
\begin{array}{rcrrrrrrrrl}
(r_n)&=&1,&1,&3,&5,&6,&6,&6,&6,&\dots\mbox{ and}\\
(t_n)&=&1,&1,&3,&13,&8,&8,&8,&8,&\dots\\
\end{array}
$$
can be realized as sequences of sum indecomposable elements of a sum closed permutation class.  Proposition~\ref{prop-gr-intervals} then implies that every growth rate between
$$
\lambda=\gr\left(\frac{1}{1-\sum r_nx^n}\right)
$$
and the greatest positive root of $x^5-2x^4-2x^2-10x+5$, the growth rate for $1/(1-\sum t_n x^n)$, is the growth rate of a permutation class, as desired.
\end{proof}

\begin{proposition}\label{prop-gr-interval2}
Every real number between the greatest positive root of $x^5-2x^4-2x^2-10x+5$ ($\approx 2.69284$) and $\lambda+1$ is the growth rate of a permutation class.
\end{proposition}
\begin{proof}
Following the technique of the previous proposition, we set $A=U^{12,12}\cup U^{12,21}\cup U^{21,12}$.  The sequence of sum indecomposable permutations properly contained in any member of $A$ is given by $1,1,3,7,8,8,8,\dots$.  We begin by defining $A^{(1)}$ to consist of $A$ together with all sum indecomposable permutations of length at most $5$.  It follows that every sequence of sum indecomposable permutations between $(r_n)=1,1,3,13,8,8,8,8,\dots$ and $(t_n)=1,1,3,13,71,11,11,11,\dots$ can be realized as the sequence of sum indecomposable permutations in a sum closed permutation class.  This shows, via Proposition~\ref{prop-gr-intervals}, that every real number between the greatest positive root of $x^5-2x^4-2x^2-10x+5$ and $3.03024$ is the growth rate of a permutation class.

For the next construction, we let $A^{(2)}$ consist of $A$ together with all sum indecomposable permutations of length at most $6$.  This shows that every sequence of sum indecomposable permutations between $(r_n)=1,1,3,13,71,8,8,8,8,\dots$ and $(t_n)=1,1,3,13,461,11,11,11,\dots$ can be realized as the sequence of sum indecomposable permutations in a sum closed permutation class, and so every real number between $3.02440$ and $3.41108$ is the growth rate of a permutation class.  Finally, we let $A^{(3)}$ consist of $A$ together with all sum indecomposable permutations of length at most $7$, which gives (as in the previous two cases) that every real number between $3.41035$ and $3.79450>\lambda+1$ is the growth rate of a permutation class.
\end{proof}

Before proving our main theorem, we need one more notion.  Given two permutation classes $\C$ and $\D$, their {\it horizontal juxtaposition\/}, $\hjuxta{\C}{\D}$, consists of all permutations $\pi$ that can be written as a concatenation $\sigma\tau$ where $\sigma$ is in the same relative order as a permutation in $\C$ and $\tau$ is in the same relative order as a permutation in $\D$.  Proofs of the following proposition can be found in numerous places, e.g., Albert~\cite{albert:on-the-length-o:}.

\begin{proposition}\label{prop-gr-juxta}
If the permutation classes $\C$ and $\D$ both have growth rates, then the growth rate of $\hjuxta{\C}{\D}$ exists and is equal to the sum of the growth rates of $\C$ and $\D$.
\end{proposition}
\begin{proof}
Every $\pi\in\hjuxta{\C}{\D}$ can be written as a concatenation $\sigma\tau$ for $\sigma\in\C$ and $\tau\in\D$ in between $1$ and $n+1$ ways, and thus
$$
\left|\hjuxta{\C}{\D}_n\right|\le \sum_{k=0}^n {n\choose k}|\C_k||\D_{n-k}|\le (n+1)\left|\hjuxta{\C}{\D}_n\right|.
$$
It follows that $\gr(\hjuxta{\C}{\D})=\gr(\C)+\gr(\D)$, as desired.
\end{proof}

The proof of Theorem~\ref{thm-gr-all-large} now follows quickly.  Propositions~\ref{prop-gr-interval1} and \ref{prop-gr-interval2} show that the set of growth rates of permutation classes contains the interval $[\lambda,\lambda+1]$.  Proposition~\ref{prop-gr-juxta} shows that $\gr\left(\hjuxta{\C}{\Av(21)}\right)=\gr(\C)+1$, so the set of growth rates of permutation classes also contains the intervals $[\lambda+1,\lambda+2]$, $[\lambda+2,\lambda+3]$, and so on.

\section{The Least Accumulation Point From Above}

As discussed in the first section, the growth rates between $\kappa$ and $\lambda$ remain a mystery.  One of the things that would be nice to pin down is the appearance of the first accumulation point from above.  Our antichains and Proposition~\ref{prop-gr-perfect} give a bound on this, improving on the bound given by Albert and Linton~\cite{albert:growing-at-a-pe:}:

\begin{proposition}
Let $\xi\approx 2.30524$ denote the unique positive root of $x^5-2x^4-x^2-x-1$ and $\zeta\approx 2.32331$ the unique positive root of $x^6-x^5-2x^4-x^3-2x^2-3x-1$.  The set of growth rates of permutation classes in the interval $[\xi,\zeta]$ contains a perfect set, and thus $\xi$ is an accumulation point from above.
\end{proposition}
\begin{proof}
Let $A$ denote the members of $U^{12,12}$ of odd length.  The set of sum indecomposable permutations properly contained in an element of $A$ is counted by the sequence $(r_n)=1,1,2,3,4,4,4,\dots$.  Thus all sequences $(s_n)$ satisfying $(r_n)\preceq (s_n)\preceq (t_n)$ for
$$
\begin{array}{rcrrrrrrrrrrrr}
(t_n)&=&1,&1,&2,&3,&5,&4,&5,&4,&5,&4,&\dots
\end{array}
$$
can be realized as sequences of sum indecomposable elements of a sum closed permutation class.  The claim then follows from Proposition~\ref{prop-gr-perfect}.
\end{proof}

We conjecture that this construction is best possible:

\begin{conjecture}
The least accumulation point from above in the set of growth rates of permutation classes is $\xi$.
\end{conjecture}




\bibliographystyle{acm}
\bibliography{../refs}

\end{document}